\documentclass[a4paper,reqno,index,11pt]{amsart}

\usepackage[normalem]{ulem} 

\usepackage[baseline]{euflag}

\DeclareMathAlphabet{\lcal}{U}{dutchcal}{m}{n}

\usepackage{amsmath, amsfonts, amsthm, amssymb}
\usepackage{mathrsfs}
\usepackage{fancyhdr}

\usepackage{enumerate}   
\usepackage{blkarray} 
\usepackage{mathrsfs}
\usepackage[title]{appendix}
\usepackage{amsthm}
\usepackage{xcolor}
\usepackage{graphicx}
\usepackage[normalem]{ulem}
\usepackage{morefloats}
\usepackage{float}
\usepackage{hyperref}
\usepackage{mathtools}

\usepackage[shortlabels]{enumitem}

\usepackage{setspace}

\usepackage[margin=2.5cm]{geometry}
\newcommand{\df}[1]{{\textit{#1}}{\index{#1}}}

 \numberwithin{equation}{section}

\allowdisplaybreaks

\makeindex

\hypersetup{
    bookmarks=true,         
    unicode=false,          
    pdftoolbar=true,        
    pdfmenubar=true,        
    pdffitwindow=false,     
    pdfstartview={FitH},    
    pdftitle={My title},    
    pdfauthor={Author},     
    pdfsubject={Subject},   
    pdfcreator={Creator},   
    pdfproducer={Producer}, 
    pdfkeywords={keyword1, key2, key3}, 
    pdfnewwindow=true,      
    colorlinks=false,       
    linkcolor=green,          
    citecolor=green,        
    filecolor=green,      
    urlcolor=green,           
    urlbordercolor={1 1 1}  
}

\newtheorem{theorem}{Theorem}[section]
\newtheorem*{theorem*}{Theorem}
\newtheorem*{conjecture*}{Conjecture}
\newtheorem{lemma}[theorem]{Lemma}

\newtheorem{proposition}[theorem]{Proposition}
\newtheorem{remark}[theorem]{Remark}
\newtheorem{corollary}[theorem]{Corollary}

\DeclareMathOperator{\range}{range}
\newcommand{\CC}{\mathbb{C}}
\newcommand{\DD}{\mathbb{D}}
\newcommand{\TT}{\mathbb{T}}

\newcommand{\ol}[1]{\overline{{#1}}}
\newcommand{\gu}{\mu}

\newcommand{\cA}{\mathcal{A}}
\newcommand{\N}{\mathbb{N}}
\newcommand{\Z}{\mathbb{Z}}
\newcommand{\R}{\mathbb{R}}
\newcommand{\C}{\mathbb{C}}

\newcommand{\htop}{\operatorname{ht}}
\newcommand{\Tt}{{\mathsf T}}

\newcommand{\Mat}[2]{\operatorname{M}_{#1}(#2)}
\newcommand{\Sym}[2]{\operatorname{Sym}_{#1}(#2)}

\begin{document}

\title[Non-negative operator-valued bivariate trigonometric polynomials]{Addendum to ``Factoring non-negative operator valued trigonometric polynomials in two variables''}

\author[\small Dritschel]{Michael Dritschel}
\address{School of Mathematics, Statistics and Physics,Herschel Building, University of Newcastle, Newcastle
upon Tyne NE1 7RU, UK}
\email{michael.dritschel@ncl.ac.uk}

\author[\small Klep]{Igor Klep}
\address{University of Ljubljana, Faculty of Mathematics and Physics, Jadranska 21, 1000 Ljubljana \& 
University of Primorska, Faculty of Mathematics, Natural Sciences and Information Technologies, Glagolja\v{s}ka 8, 6000 Koper, Slovenia.}
\email{igor.klep@fmf.uni-lj.si}

\author[\small McCullough]{Scott McCullough}
\address{Department of Mathematics, University of Florida, Gainesville, FL} 
\email{sam@ufl.edu} 

\author[\small Vol\v ci\v c]{Jurij Vol\v ci\v c}
\address{The University of Auckland, Department of Mathematics,
Private Bag 92019,
Auckland 1142,
New Zealand.}
\email{jurij.volcic@auckland.ac.nz}


\date{}

\begin{abstract}
    Factorization for positive semidefinite matrix-valued  polynomials over a nonsingular compact affine real surface is established. 
   Corollaries include Fej{\'e}r–Riesz factorization for bivariate matrix polynomials that take positive semidefinite values and resolutions
   to questions posed by Mehta-Slofstra-Zhao and Savchuk-Schm\"udgen.  An explicit example shows a conclusion of 
   \cite[Theorem, p.~519]{Dritschel2025} that arises organically from its proof need not hold. The difficulty is traced to \cite[Theorem~3.7]{Dritschel2025}.
\end{abstract}

\maketitle

\section{Introduction}
  The main result of this note, Theorem~\ref{thm:main},  is a factorization for positive semidefinite matrix-valued  polynomials over a nonsingular compact affine surface over $\R.$  A principal corollary is Fej{\'e}r–Riesz factorization for bivariate matrix polynomials that take positive semidefinite values.
  
 \begin{theorem}
  \label{thm:main:intro}
   Given a positive integers  $d$  and $n$  and $n\times n$ complex matrices $A_{j,k}$ for  $-d\le j,k\le d,$ if the the $n\times n$ matrix-valued Laurent polynomial
  \[
   A(z,w)=\sum_{j,k=-d}^d  A_{j,k} z^j w^k 
  \]
  satisfies, $A(z,w)\succeq0$ for all $|z|=1=|w|,$ then there exist positive integers $e$ and $m$ and $m\times n$ matrices $B_{j,k}$ for $0\le j,k\le e$ such that $A(z,w)=B(z,w)^* B(z,w)$ for all $|z|=1=|w|,$ where $B$ is the analytic matrix-valued polynomial
 \[
  B(z,w) = \sum_{j,k=0}^e B_{j,k} z^j w^k.
 \]
 \end{theorem}
 
 The article contains two further corollaries to Theorem~\ref{thm:main}. Corollary~\ref{c:msz}  answers an  open problem on positivity in tensor products from \cite{Mehta2026} that arose in quantum information theory.  Corollary~\ref{c:mtxsphere} asserts that $\Mat{n}{\C[S^2]}$ admits a perfect Positivstellensatz, answering 
 a question from \cite[Section 11]{Savchuk2012}. 
 
 Theorem~\ref{thm:main:intro} is mute on degree bounds.  It is known that, even in the scalar case,  for a given $d,$ no bound on $e$ in Theorem~\ref{thm:main:intro} exists \cite[Corollary 5.5]{Scheiderer2005}. In particular, Theorem~\ref{thm:main:intro} can not hold in general with operator coefficients $A_{j,k}.$  While no bound on $e$ exists, it is still possible certain partial degree bounds can hold. 
However,  it is shown via an explicit example that the degree bound  in the statement of bivariate matrix-valued Fej{\'e}r–Riesz factorization \cite[Theorem~4.1]{Dritschel2025} that  arises organically from its proof need not hold.   As a related example shows, the difficulty in the proof of \cite[Theorem~4.1]{Dritschel2025} lies in \cite[Theorem~3.7]{Dritschel2025}.  These examples appear in Section~\ref{s:examples}.   The main result is developed in Section~\ref{s:main-result}. The three corollaries are found in Section~\ref{s:cors}. 
While the examples above by no means rule out an operator-theoretic proof of Theorem \ref{thm:main:intro} in the vein of \cite{Dritschel2025}, in the interest of correcting the record in a timely fashion, the proof given here relies on results and methods of commutative algebra and real algebraic geometry, which also generalizes to other nonsingular compact surfaces in Theorem~\ref{thm:main}.

 \subsection*{Use of AI}
The authors acknowledge the use of GPT-5.6 Sol  and GPT-5.5 during the preparation of this note. 
Following the authors' identification of a flaw in the original proof, the model was utilized to explore and construct explicit counterexamples to intermediate statements and to suggest candidate arguments for the revised proof of Theorem \ref{thm:main}. 
All counterexamples, proofs, and mathematical details were independently checked, verified, and formalized by the authors, who assume full responsibility for their mathematical correctness.

\subsection*{Acknowledgments} 
{\small 
This work was performed within the project COMPUTE, which is funded within the QuantERA II Programme that has received funding from the EU's H2020 research and innovation programme under the GA No.~101017733 {\normalsize\euflag}. 
IK also acknowledges support of the Slovenian Research Agency program P1-0222 and grants J1-50002, N1-0217, J1-60011, J1-50001, J1-3004 and J1-60025.
JV was supported by the Marsden Fund MFP-UOA2528 from Government funding,administered by the Royal Society Te Ap\=arangi, and the New Staff grant 3734441 by the Faculty of Science, the University of Auckland.
}

\subsection*{Thanks}
The authors thank Greg Blekherman, Abhay Jindal, Michael Jury, James Pascoe, Daniel Plaumann, Rainer Sinn and Aljaž Zalar for discussions.  The second and third authors also thank Jacob Levenson for discussions of and related to the many fruitful ideas in \cite{Dritschel2025} that led to the article \cite{KLM} and a follow-up currently in preparation.

\section{Counterexamples} 
\label{s:examples}

This section develops counterexamples to the results in \cite{Dritschel2025}. In the first subsection we present an example showing that the conclusion of \cite[Theorem 3.7]{Dritschel2025} need not hold. The second subsection then develops a counterexample to the degree bound of the main result, \cite[Theorem 4.1]{Dritschel2025}.

\subsection{Theorem 3.7}
    There is  an explicit verifiable counterexample to \cite[Theorem~3.7]{Dritschel2025} even for the case where all the coefficient Hilbert spaces are the complex numbers (one dimensional).  Accordingly we restrict out attention to the case $\mathfrak{K}=\mathfrak{H}=\C.$ To describe this setting, we use the standard notation $L^\infty=L^\infty(\TT)$ for the space of bounded measurable complex-valued functions on the unit circle $\TT$  with respect to Lebesgue measure.  Likewise, $L^2=L^2(\TT)$    and $H^2=H^2(\TT)$ denote the usual Hilbert space and Hardy-Hilbert space, respectively, of $L^2$ functions on the unit circle.  In particular, $H^2$ consists of those $L^2$ functions with vanishing negative Fourier coefficients; equivalently, those $L^2$ functions that are the boundary values of functions analytic in the unit disk $\DD.$   An element $\varphi \in L^\infty$ induces an operator (multiplication operator) $M_\varphi:L^2\to L^2$ by $M_\varphi f=\varphi f,$ known as a bilateral Toeplitz operator.   Letting $P_+:L^2\to H^2$ denote the orthogonal projection, the operator $T_\varphi:H^2\to H^2$ defined by
   \[
     T_\varphi  f= P_+ \varphi f,
   \]
 is the \df{unilateral Toeplitz operator with symbol $\varphi$}.   \index{$T_\varphi$}
 The operator $M_\varphi$ is the bilateral version of $T_\varphi.$ 
 
 Turning to the concrete example, fix $0<r<1$ and let $a=2(1+r)$ and $b=1+rz$ and set $A=T_a$ and $B=T_b.$ In the notation of \cite{Dritschel2025}, the Toeplitz operator $B$ has degree $(1,0)$ since $b,$ written as $b_+ + b_-$, where $b_-$ is a conjugate analytic polynomial and $b$ is an analytic polynomial, the degrees of $b_+$ and $b_-$ are $1$ and $0$,  respectively.   Let $M=\frac12 A,$ and consider the  block $2\times 2$ matrix,
 \[
  \begin{pmatrix} A-M & B \\ B^* & M  \end{pmatrix} \, : \, \begin{matrix} H^2 \\ \oplus \\ H^2\end{matrix} \mapsto \begin{matrix} H^2 \\ \oplus \\ H^2\end{matrix}.
    \]
    It is positive semidefinite (denoted ${}\succeq0$), since pointwise for $|z|=1,$ 
\[  
     \begin{pmatrix}  (1+r) & 1+rz \\ 1+ r\ol{z} & (1+r) \end{pmatrix} \succeq0.
 \]
 Thus the pair $(A,B)$ satisfies the hypotheses of \cite[Theorem~3.7]{Dritschel2025}.  
 
 Let $m=|b|$ and observe $0\preceq 2T_m\preceq T_a$ since $a \ge 2m\ge 0$ pointwise.  On the other hand, if $T=T_f$ is a Toeplitz operator and 
 there a $g$ such that $0\le T_g \le T_a,$ and
 \[
    \begin{pmatrix}  T_f - T_g & T_b \\ T_b^* & T_g  \end{pmatrix} \succeq 0,
 \]
 then $f\ge g\ge 0$ and at the same time $m^2=|b|^2 \le (f-g) g.$ Thus $f \ge 2m.$  Equivalently, $T_f \ge 2T_m.$
 It follows that, in the statement of \cite[Theorem~3.7]{Dritschel2025}, 
 $M\in \mathscr{M}\ne \emptyset$  and
  $\widehat{A} = 2T_m.$  Moreover, with this choice
 of $f=2m,$ the only choice for $g$ is  $m.$ Hence, in the statement of \cite[Theorem~3.7]{Dritschel2025}, $\widehat{M}=T_m$
 so that $\widehat{A}-\widehat{M}=\widehat{M} = T_m.$
 To see that the claimed existence of Toeplitz operators $R_E=V_E E$ and $R_F=V_F F$ (the exact description of $E,F$ and $V_E, V_F$
 does not play a role) satisfying
   \begin{equation}
   \label{e:RFstarRE}
      \widehat{M} =  R_F^* R_F = T_m = R_E^* R_E = \widehat{A}-\widehat{M}, \  \  \  B= R_F^* R_E =T_b 
   \end{equation}
    leads to a contradiction,  let  $u$ and $v$ denote the symbols of $R_E$ and $R_F$, respectively. 
    In brief, from equation~\eqref{e:RFstarRE} it follows that $u^*u=v^*v=m$ and $|u^*v|= |b|=m$. Thus $u$ and $v$ are (pointwise a.e.) collinear,  which leads to a contradiction (see the details below).

     To pinpoint the difficulty in the proof of \cite[Theorem 3.7]{Dritschel2025} and using the notation therein,  where $\mathfrak{E}=\mathfrak{F}=\CC$ and $S_E=S_F=S$ are the unilateral shifts, the coefficients of  $g_n=(I-S^n S^{*n}) G S^n h_{1,n}$ beyond $n$ are $0;$ and not that the first $n$ coefficients of $g_n$ are $0,$ as claimed.  Consequently, the claim that $(g_n)_n$ converges to $0,$ and hence $f_2=0,$  is not justified.
 This claim is then  used to show $S_W$ is a shift, where  $W$ and $S_W$ are constructed during the proof.  We next describe this
 construction and show, in the present example, 
  $S_W$ is not a shift, using the fact $S_W$ is a shift if and only if 
 \begin{equation}
 \label{e:SW-no-shift}
  \cap_{n=0}^\infty \widetilde{S}^n \ol{\range{W}} = \{0\},
 \end{equation}
  where $\widetilde{S}=M_z$ is the bilateral shift of multiplication by $z$ on $L^2.$

The two positive Toeplitz operators $R_E^*R_E$ and $R_F^*R_F$  are the same with symbol $m.$  The polynomial $b$ has no zeros in a neighborhood of the closed unit disk and hence admits an (analytic) square root  $q.$ Thus,
\[
|q|^2=m=|b|
\]
Choose
\(
E=F=T_q.
\)
Let $G=V_F^*V_E$. Then $G=T_u$ is a unilateral Toeplitz operator satisfying
\[
B=F^*GE
\]
as in \cite[Theorem 3.7]{Dritschel2025}.
Therefore
\[
T_b=T_q^*\,G\,T_q  = T_q^* T_\gu T_q = T_{\ol{q} \gu q}.
\]
 Thus, on the boundary of the unit disk, 
\[
b=\overline q\, \gu\, q = \gu |q|^2 = \gu m.
\]
Since \(m=|b|\), 
\[
\gu=\frac{b}{m}=\frac{b}{|b|} 
\]
a.e. on the boundary of the disk. Since \(|\gu|=1\) a.e., the bi-infinite multiplication operator \(\widetilde G=M_\gu\) is unitary, as claimed in the proof of \cite[Theorem 3.7]{Dritschel2025}. 

Express $L^2 = H^2_{-}\oplus H^2,$ where $H^2_{-}$ is the orthogonal complement of $H^2$ in $L^2.$ 
The Hankel operator $H_\gu$ with symbol $u$ is the operator $H_\gu:H^2\to H^2_{-}$ defined by
\[
 H_u f= P_{-} uf, 
\]
 where $P_{-}$ is the orthogonal projection of $L^2$ onto $H^2_{-}.$ 
The operator $W:L^2\to H^2\oplus H^2$ is given in block form as
\[
 \begin{pmatrix} D & 0 \\G & I \end{pmatrix} = \begin{pmatrix} H_\gu & 0 \\ T_u & I \end{pmatrix} : \begin{matrix} H^2 \\ \oplus \\ H^2 \end{matrix} \mapsto \begin{matrix} H^2_{-} \\ \oplus \\  H^2 \end{matrix}.
\]

The operator $S_W$ is the Lowdenslager isometry associated to $W.$ What is important here is that $S_W$ is unitary, and not a shift, if the range of $W$ is dense. 
Indeed, in this case the intersection in equation~\eqref{e:SW-no-shift} is all of $L^2.$  
To prove that $\range W$ is dense, it suffices to show 
 \(D H^2\) is dense in \(H^2_-\). Equivalently,  it suffices to show if $h\in H^2$ and $uh\in H^2,$ then $h=0.$  Suppose $h\ne 0.$ Since \(|\gu|=1\), \(h\) and \(\gu h\) would have the same boundary modulus, their outer factors are the same.  Letting $\varphi$ and $\psi$ denote the inner factors of $\gu h$ and $h$ respectively,
 \[
  \gu\psi = \varphi
 \]
  a.e. on the boundary of the disk.\footnote{So $u$ is a function of bounded type.} Hence the same is true of
  \[
    \gu^2 \psi^2 =\varphi^2.
  \]
  Since 
  \[
    \gu^2 = \frac{\ b\ }{\ \ol{b} \ }=  \frac{1+rz}{r+z}
  \]
   on the boundary of the disk, 
\[
   (1+rz) \psi^2 = (r+z) \varphi^2
\]
  on the boundary of the disk, and since both $(1+rz)\psi^2$ and $(r+z)\varphi^2$ are analytic in the disk, the are equal in the disk. But the right hand side has a zero of odd order at $-r;$ whereas the left hand side has a zero even order at $-r,$ a contradiction that shows $h=0$ and completes the proof that $S_W$ is not a shift.  It is essentially this same argument that shows colinearity  of $u$ and $v$ together with $|u|=|v|$,  which gives $g = \frac{u}{v}$ 
  is unimodular, that produces the contradiction above.

\subsection{A counterexample to degree bounds}

This section presents a example where the degree bound $2d_1-1$ in \cite[Theorem~4.1]{Dritschel2025} does not hold.
 Fix $0<\rho<1$ and let $p$ and $q$ denote the analytic polynomials in the complex variable $w,$
\[
p(w) = w - \rho, \qquad q(w) = 1 - \rho w. 
\]
Thus, 
\[
 \theta(w) = \frac{p(w)}{q(w)}
\]
 is the Blaschke factor with zero at $\rho$ and so is unimodular on the boundary $\TT$ 
 of the unit disk $\DD$ in the complex plane $\CC.$ Let
 \[
 h(w) \coloneqq q(w)^* q(w) = (1 - \rho w^{-1})(1 - \rho w) = \frac{p(w)q(w)}{w},
 \]
  and note, for $|w|=1,$
 \[
  h(w) = 1+\rho^2 -\rho w-\rho {w^{-1}} \in \CC[w,w^{-1}].
 \]
 Generally we identify hermitian squares of analytic polynomials as elements of $\CC[w,w^{-1}]$, where
 $w\in \TT$ is implicit.
 
 Let 
\begin{equation}
\label{e:P}
P(z,w) = h(w) I_4 + z B(w) + {z^{-1}}  B(w)^*,
\end{equation}
where, 
\begin{equation*}
B(w) = \frac{1}{\sqrt{2}w}
\begin{pmatrix}
0 & 0 & 0 & p^2 \\
q^2 & 0 & p^2 & 0 \\
0 & 0 & 0 & -q^2 \\
0 & 0 & 0 & 0
\end{pmatrix}.
\end{equation*}
Thus $\deg_z P = \deg_w P=1$, and $B$ is a matrix-valued   polynomial in $w$ of degree $1$.

The next proposition provides the counterexample to the degree bound in \cite[Theorem 4.1]{Dritschel2025}.

\begin{proposition}
 \label{prop:main-cex}
   With the  above notation,  $P=G^*G$  where $G$ is a $2\times 4$ matrix-valued analytic polynomial with $\deg_z G=2,$
   \begin{equation}
G(z,w) = 
\begin{pmatrix}
\dfrac{qz}{\sqrt{2}} & \dfrac{p+q}{2} & \dfrac{pz}{\sqrt{2}} & \dfrac{p-q}{2}z^2 \\[3mm]
\dfrac{qz}{\sqrt{2}} & \dfrac{p-q}{2} & -\dfrac{pz}{\sqrt{2}} & \dfrac{p+q}{2}z^2
\end{pmatrix}.
\end{equation}
 In particular $P(z,w)\succeq0$
for all $|z|=1=|w|.$ 

Further, if $n$ is a positive integer and $F$ is a $n\times 4$ matrix-valued analytic polynomial such that 
$P=F^*F,$ then $\deg_z F\ge 2.$
\end{proposition}

Before turning to the proof proper we introduce some notation and record a lemma.  Let
\[
 B_0 = \frac{1}{\sqrt{2}} \begin{pmatrix} 0&0&0&1\\1&0&1&0\\0&0&0&-1\\0&0&0&0\end{pmatrix}, \ \ \
 M_0=\frac12 \begin{pmatrix}  1&0&1&0\\0&0&0&0\\1&0&1&0\\0&0&0&1\end{pmatrix}.
\]
In terms of the orthonormal basis 
\[
u_1 = e_2, \qquad v_1 = \frac{e_1 + e_3}{\sqrt{2}}, \qquad u_2 = \frac{e_1 - e_3}{\sqrt{2}}, \qquad v_2 = e_4
\]
of $\mathbb{C}^4,$ 
\[
B_0 = u_1 v_1^* + u_2 v_2^*, \qquad M_0 = v_1 v_1^* + v_2 v_2^*
\]
and 
\begin{equation}
\label{e:B0}
B_0^* B_0 = M_0, \qquad B_0 B_0^* = I - M_0, \qquad B_0^2 = 0, \qquad B_0^* u_j= v_j.
\end{equation}

The following lemma exposes the key uniqueness feature of this example.

\begin{lemma}
\label{l:lem:M}
If $M$ is a $4\times 4$ matrix and 
\begin{equation}
\label{e:lem:M}
J = \begin{pmatrix}
I - M & B_0 \\
B_0^* & M
\end{pmatrix} \succeq 0,
\end{equation}
then 
\(
M = M_0.
\)
\end{lemma}

\begin{proof}
For $j=1,2$, put
\[
x_j(z)=  -z u_j + v_j 
\]
and compute   $(zB_0 + {z^{-1}} B_0^*)x_j(z)=-x_j(z)$   for $|z|=1$, so that 
\[
 ( I + zB_0 +{z^{-1}}B_0^* ) x_j(z) =0.
\]
Hence 
\[
\langle  J \gamma_j,\, \gamma_j \rangle 
 =  \langle  ( I + zB_0 + {z^{-1}}B_0^* )x_j ,x_j \rangle  =0,
\]
where
\[
 \gamma_j(z) = \begin{pmatrix} x_j(z) \\z x_j(z) \end{pmatrix}.
\]
 Since the matrix $J$ (of equation~\eqref{e:lem:M}) is positive semidefinite and $J\gamma_j=0,$ it follows that, for $|z|=1,$ 
\[
  (B_0^* + z M)x_j(z) =0. 
\]
On the other hand, from equation~\eqref{e:B0},
\[
 B_0^* x_j = -z v_j
\]
 and thus,  $Mx_j(z)= v_j;$ that is,
 \[
  Mv_j = v_j = M_0 v_j , \qquad M u_j=0 = M_0 u_j. 
 \]
  Because $\{u_1,u_2,v_1,v_2\}$ is a basis of $\CC^4,$  it follows
  that $M=M_0.$
 \end{proof}

\begin{proof}[Proof of Proposition~\ref{prop:main-cex}]
 Let 
\[
 S(w) = \begin{pmatrix} \frac{q}{\theta} &0&0&0 \\ 0 & q &0&0\\0&0&p&0\\0&0&0&q\end{pmatrix}.
\]
   Using the relations
    \begin{equation}
    	   \label{e:pandq}
  		p = w \ol{q},\quad
  		h =\ol{q} q = \frac{pq}{w},\quad
  		\ol{q}p = h \theta  = \frac{p^2}{w},\quad
  		q\ol{p}  =  h\theta^{-1}  = \frac{q^2}{w},  
  \end{equation}
  valid for $|w|=1,$ direct computation shows, for $|z|=1=|w|,$ that  $F^*F=P.$ Further, the function $S$ satisfies $S(w)^*S(w)=h I_4$ and 
\begin{equation}
\label{e:SB0S}
 S(w)^* B_0 S(w) = \frac{1}{\sqrt{2}w} \,  \begin{pmatrix} 0&0&0&h\theta \\ \frac{h}{\theta} & 0& \frac{h}{\theta} & 0 \\ 0&0&0 & h\theta 
 \\ 0&0&0&0  \end{pmatrix} = B(w).
\end{equation}
 
  Arguing by contradiction,  suppose there is an $n$ and an $n\times 4$ matrix-valued analytic trigonometric polynomial
 \[
  F(z,w) = F_0(w) + zF_1(w)
 \]
  such that $P=F^*F.$   Comparing coefficients in equation~\eqref{e:P} to those of $F^*F$ gives,
\[
F_0^* F_0 = hI - M, \qquad F_0^* F_1 = B,
\]
where $M= F_1^* F_1.$ In particular,
\begin{equation}
\label{e:Mjk}
  M_{j,k} = (F_1^* F_1)_{j,k} \in \mathbb{C}[w, w^{-1}]
\end{equation}
 for $1\le j,k\le 4$ and $|w|=1.$
Let 
\begin{equation*}
G(w) \coloneqq
\begin{pmatrix}
hI - M & B \\
B^* & M  
\end{pmatrix}  = \begin{pmatrix}  F_0^* \\ F_1^* \end{pmatrix} \, \begin{pmatrix} F_0 & F_1 \end{pmatrix} \, \succeq 0.
\end{equation*}
For $|w|=1,$ 
using equation~\eqref{e:SB0S},
\[
 0\preceq  \begin{pmatrix} S^{-\ast}(w) &0\\ 0 &S^{-\ast}(w) \end{pmatrix} \, G \,  \begin{pmatrix} S^{-1}(w) &0\\ 0 &S^{-1}(w) \end{pmatrix} 
  = 
\begin{pmatrix} I - \widetilde{M} & B_0 \\
B_0^* & \widetilde{M}
\end{pmatrix} \succeq 0,
\]
where $\widetilde{M} = S^{-*} M S^{-1}.$
Lemma~\ref{l:lem:M} gives, for $|w|=1,$
\[
\widetilde{M}(w)= M_0,
\]
hence
\begin{equation}
\label{e:M=SMS}
F_1^*F_1 = M(w) = S^*(w) M_0 S(w)  = \frac12 \,   \begin{pmatrix}  h & 0 & \frac{p^3}{w q} &0 \\0&0&0&0\\ 
\frac{w \ol{p}^3}{\ol{q}}
 & 0& p\ol{p} & 0\\0&0&0& q\ol{q}\end{pmatrix}
\end{equation}
and, comparing with equation~\eqref{e:Mjk},  we have arrived at the contradiction $M_{1,3} \notin \mathbb{C}[w, w^{-1}].$
 \end{proof}

\section{An exact matricial Positivstellensatz on compact smooth real surfaces}
\label{s:main-result}
In this section we state and prove a sum-of-squares theorem for matrices over polynomials on compact smooth real surfaces. In particular, this theorem implies the matricial bivariate Fej{\'e}r–Riesz theorem; see Corollary \ref{c:mtx2FR} below.

\begin{theorem}
\label{thm:main}
Let $X$ be a nonsingular irreducible affine surface over $\R$ such that $X(\R)$ is compact, and let $R$ be the ring of real polynomial functions on $X$. If $A=A^*\in\Mat{d}{\C}\otimes_{\R}R$ is such that $A(p)\succeq0$ for all $p\in X(\R)$, then there exist $m\in\N$ and $B\in \Mat{m\times d}{\C}\otimes_{\R}R$ such that $A=B^*B$.
\end{theorem}

\begin{remark}
It suffices to prove the statement for real symmetric matrices over $R$: if $A\in\Sym{d}{R}$ and $A\succeq0$ on $X(\R)$, then $A=B^\Tt B$ for some $B\in \Mat{m\times d}{R}$.

Indeed, suppose a hermitian $\mathbf{A}\in \Mat{d}{\C}\otimes_{\R}R$ satisfies $\mathbf{A}\succeq0$ on $X(\R)$. There is a unitary matrix $U\in\Mat{2d}{\C}$ such that
$$A:=U\begin{pmatrix}
\mathbf{A}&0 \\ 0& \overline{\mathbf{A}}
\end{pmatrix}U^*\in \Sym{d}{R}.$$
By the real statement above, $A=B^{\Tt} B$ for some $B\in \Mat{m\times 2d}{R}$. Let $E=(I\ 0)\in\Mat{d\times 2d}{\R}$; then $\mathbf{A} = EU^*AUE^\Tt=(BUE^\Tt)^*(BUE^\Tt)$.

Thus, we restrict to the real symmetric case from hereon.
\end{remark}

Let $X$ be as in Theorem \ref{thm:main}; the proof of (the real symmetric version of) Theorem \ref{thm:main} below relies heavily on commutative algebra, so we start by underlining the properties of the ring $R$. Since $X\subset\C^n$ is defined over $\R$, its vanishing ideal is of the form $\C\otimes_\R J$ for an ideal $J\subset \R[x_1,\dots,x_n]$, and $R=\R[x_1,\dots,x_n]/J$. Then, the ring $R$ is a noetherian domain (since $X$ is irreducible) of dimension 2 (since $X$ is a surface). Since $X$ is nonsingular, $R$ is regular, in the sense that all its localizations at prime ideals are regular local rings. Note that since $R$ is a domain of dimension 2, it has three kinds of prime ideals: $\{0\}$, the maximal ideals (corresponding to points in $X$), and the remaining ones, corresponding to irrreducible curves in $X$, which have height one (see \cite[Section 5]{Matsumura1989} for more on height). The latter height-one primes are crucial in the proof. 
Let $K$ be the field of fractions of $R$; note that $K$ is formally real, in the sense that $-1$ is not a sum of squares in $K$.
If $\mathfrak{p}$ is a height-one prime of $R$, then $R_{\mathfrak p}$ is a regular local ring of dimension 1, and thus a discrete valuation ring by \cite[Theorem 11.2]{Matsumura1989}.
Furthermore, within $K$,
\begin{equation}\label{eq:krullintersection}
R=\bigcap_{\htop\mathfrak p=1}R_{\mathfrak p}
\end{equation}
by the localization intersection property for normal (and in particular, regular) noetherian domains \cite[Theorem 11.5]{Matsumura1989}.

Given symmetric $d\times d$ matrix $A$ over $R$ that is positive semidefinite, finding a factorization $A=B^\Tt B$ over $R$ is guided by the following steps:
\begin{enumerate}
	\item $A$ admits a factorization $T^\Tt T$ with $m\times d$ matrix over the field $K$. More generally, for each curve on $X$, one can find such a factorization that is regular along the curve. This step relies on existing 1-dimensional sum-of-squares results, and diagonalizations over discrete valuation rings.
	\item The image of $R^d$ in $K^m$ under $T$ lies in an $R$-module $L$ with two distinguished properties: $L$ is projective, and the standard symmetric bilinear form on $L$ has values in $R$. The module $L$ is obtained by patching (akin to \eqref{eq:krullintersection}) of the local modules associated with factorizations along the curves from the previous step.
	\item Since $L$ is projective, the standard form extends to a point-wise positive definite form on a free $R$-module, i.e., it is given by a positive definite matrix over $R$; then a standard matricial Positivstellensatz on compact sets implies that it is a sum of hermitian squares of linear forms. Projecting back to $L$, we obtain a decomposition of the standard form into hermitian squares of linear forms on $L$.
	\item Since the standard form on $L$ is valued in $R$, the hermitian squares of linear forms specialised to the image of $T$ are valued in $R$, and give rise to a decomposition of $A$ as a sum of hermitian squares over $R$.
\end{enumerate}

\begin{remark}
Let us compare the above sketched proof of Theorem \ref{thm:main} with the scalar case $d=1$, established in \cite[Corollary 3.6]{Scheiderer2006}. Both proofs rely heavily on positivity and sums of squares in local rings, and local-global patchings. However, \cite{Scheiderer2006} utilizes localizations at maximal ideals, while the proof below relies on localizations at height-one prime ideals; consequently, the patching procedures also differ considerably.
\end{remark}

\subsection{Rational factorizations}

We first establish rational factorizations of $A$ that are regular along given curves.

\begin{proposition}
\label{prop:localgram}
For every height-one prime $\mathfrak p\subset R$ there exist 
$m\in\N$ and $C\in \Mat{m\times d}{R_{\mathfrak p}}$ such that
$A=C^{\mathsf T}C$.
\end{proposition}

\begin{proof}
Since $R_{\mathfrak p}$ is a discrete valuation ring, every symmetric matrix over it is congruent (over $R_{\mathfrak p}$) to a diagonal matrix. 
Namely, there exists $P\in \operatorname{GL}_d(R_{\mathfrak p})$ such that $P^\Tt AP$ is a diagonal matrix. If $a$ is its diagonal entry, then $a=u^\Tt Au$ for some $u\in R_{\mathfrak p}^d$. 
Since $A$ is positive semidefinite on $X(\R)$, $a$ is a positive semidefinite element in $R_{\mathfrak p}$, and therefore a sum of squares in $R_{\mathfrak p}$ by \cite[Theorem 3.9]{Scheiderer2001} since $R_{\mathfrak p}$ is a discrete valuation ring (alternatively, $s^2u=v \in R^d$ for some $s\in R\setminus\mathfrak{p}$, then $s^2a=v^\Tt Av\in R$ is nonnegative on $X(\R)$, and thus a sum of squares in $R$ by \cite[Corollary 3.4]{Scheiderer2006}, so $a$ is a sum of squares in $R_{\mathfrak p}$).
Consequently, the diagonal matrix $P^\Tt AP$ factors as $B^\Tt B$ for some $B\in \Mat{m\times d}{R_{\mathfrak p}}$. Then, take $C=BP^{-1}$.
\end{proof}

By Proposition~\ref{prop:localgram} applied to any height-one prime, there exist $m\in\N$ and $T\in \Mat{m\times d}{K}$ such that $A=T^\Tt T$. 
Since $R$ is a noetherian domain, then for every nonzero $a\in R$, the height-one primes containing $a$ are precisely the minimal primes containing $a$ by Krull's principal ideal theorem \cite[Theorem 13.5]{Matsumura1989}, and there are only finitely many of them by \cite[Theorem 6.5]{Matsumura1989}. By applying this fact to the denominators of entries in $T$, we see that they belong to only finitely many height-one primes $\mathfrak p_1,\dots,\mathfrak p_s$ of $R$, which we call exceptional. For each $i=1,\dots,s$, choose a local factorization
\[
A=C_i^\Tt C_i,
\qquad C_i\in M_{m_i\times d}(R_{\mathfrak p_i})
\]
from Proposition~\ref{prop:localgram}. After adding zero rows to all these matrices, we may regard $T$ and every $C_i$ as matrices with a common number of rows, $m$. 

\subsection{Patching local modules}

Equip the vector space $K^m$ with the standard bilinear form 
$\langle u,v\rangle=u^\Tt v$
and the associated quadratic form $\langle v,v\rangle$. Since $K$ is formally real, this quadratic form is anisotropic: if $\sum_j v_j^2=0$, then all $v_j=0$. 
For each exceptional prime $\mathfrak p_i$, the assignments $Tv\mapsto C_iv$ 
define an isometry between $\operatorname{im}(T)$ and
$\operatorname{im}(C_i)$. Indeed, if $\langle Tv,Tv\rangle =v^\Tt Av=\langle C_iv,C_iv\rangle$, and $\ker T=\ker C$ because the standard form is anisotropic.  By the Witt extension theorem \cite[Theorem III.4.1]{Baeza1978}, there exists an orthogonal $g_i\in \operatorname{O}_m(K)$ such that
\begin{equation}
\label{eq:giT}
g_iT=C_i.
\end{equation}

For every height-one prime $\mathfrak p$, define an
$R_{\mathfrak p}$-module in $K^m$ by
\begin{equation}
\label{eq:localL}
L_{\mathfrak p}=\begin{cases}
 g_i^{-1}R_{\mathfrak p_i}^m,&\mathfrak p=\mathfrak p_i,\\[3pt]
 R_{\mathfrak p}^m,&\mathfrak p\notin\{\mathfrak p_1,\dots,\mathfrak p_s\}.
\end{cases}
\end{equation}
Then,
\begin{equation}
\label{eq:TinlocalL}
T R_{\mathfrak p}^d\subseteq L_{\mathfrak p} \qquad \text{for every height-one }\mathfrak p.
\end{equation}
Indeed, if $\mathfrak p$ is exceptional, this holds by \eqref{eq:giT}; if $\mathfrak p$ is not exceptional, then the denominators of entries in $T$ do not lie in $\mathfrak{p}$, and are therefore invertible in $R_{\mathfrak p}$.
Moreover, every $L_{\mathfrak p}$ is self-dual for the standard pairing:
\begin{equation}
\label{eq:selfduallocal}
L_{\mathfrak p}^{\#}
 :=\{z\in K^m:\langle z,L_{\mathfrak p}\rangle
                  \subseteq R_{\mathfrak p}\}
 =L_{\mathfrak p}.
\end{equation}
This is clear for non-exceptional primes; for an exceptional prime ${\mathfrak p}_i$, it follows from orthogonality of $g_i$.

Next, define
\begin{equation}
\label{eq:globalL}
L:=\bigcap_{\htop\mathfrak p=1}L_{\mathfrak p}\subseteq K^m.
\end{equation}
By \cite[Theorem VII.4.3]{Bourbaki1989}, $L$ is a finitely generated reflexive $R$-module whose localization at each height-one prime $\mathfrak{p}$ equals
$L_{\mathfrak p}$. Here, an $R$-module $M$ has the dual $R$-module $M^*=\operatorname{Hom}_R(M,R)$; then, $M$ is reflexive if the natural $R$-module homomorphism $M\to M^{**}$ is an isomorphism.
Since $R$ is a regular ring of dimension $2$, $L$ is a finitely generated projective $R$-module by \cite[Corollary 1.4]{Hartshorne1980}.

Observe that the standard pairing on $K^m$ restricts to an $R$-valued unimodular pairing
on $L$.  Indeed, for $w,z\in L$ we have
$\langle w,z\rangle\in R_{\mathfrak p}$ for every height-one
$\mathfrak p$ by \eqref{eq:selfduallocal}; hence
$\langle w,z\rangle\in R$ by \eqref{eq:krullintersection}. Now denote
\[
L^{\#}=\{z\in K^m:\langle z,L\rangle\subseteq R\}.
\]
By the previous observation, $L\subseteq L^{\#}$; on the other hand, localization gives 
$(L^{\#})_{\mathfrak p}=L_{\mathfrak p}^{\#}=L_{\mathfrak p}$ 
for all height-one primes, and so $L^{\#}\subseteq\bigcap_{\htop\mathfrak p=1} L_{\mathfrak p}=L$. Therefore,
\begin{equation}
\label{eq:Lselfdual}
L^{\#}=L.
\end{equation}
Finally, \eqref{eq:TinlocalL} implies
\begin{equation}
\label{eq:TglobalL}
T R^d\subseteq L.
\end{equation}

\subsection{Factorization of the standard form on the patched module}

Since $L\subseteq L^{\#}$, we have a unimodular quadratic form $h:L\to R$ on the projective module $L$ given by $h(z)=\langle z,z\rangle$. 
For $p\in X(\R)$ let $\R(p)$ denote $\R$ with the $R$-module structure $a\cdot \alpha=a(p)\alpha$ (i.e., $\R(p)$ is the residue field of $R$ at $p$). 

\begin{lemma}
\label{lem:hpositive}
For $p\in X(\R)$, the induced quadratic form 
$h_p:L\otimes_R\R(p)\to\R$ 
is positive definite.
\end{lemma}

\begin{proof}
Let $z\in L$.  Viewed in $K^m$, write $z=(z_1,\dots,z_m)$ and choose a common nonzero denominator $a\in R$ such that $az_i\in R$ for all $i$.  Then
\[
a^2h(z)=\sum_{i=1}^M(az_i)^2.
\tag{7.1}
\label{eq:denomsos}
\]
At every $p\in X(\R)$ such that $a(p)\neq0$, \eqref{eq:denomsos} gives
$h(z)(p)\ge0$.  The zero set of a nonzero regular function on the real surface $X(\R)$ has
empty interior; hence the set where $a\ne0$ is dense.  By continuity,
$h(z)(p)\ge0$ for every $p\in X(\R)$. 
Since every element of the fiber $L\otimes_R\R(p)=L/\mathfrak m_pL$ has a representative in $L$, the fiber form $h_p$ is positive semidefinite.

On the other hand, \eqref{eq:Lselfdual} states that the adjoint map
$L\to L^*$ is an isomorphism.  Tensoring with the residue field at $p$ shows that $h_p$ is nondegenerate, and thus positive definite.
\end{proof}

We next convert fiber-wise positive definiteness into an algebraic Gram
factorization. The compactness of $X(\R)$ is used only in the following proposition.

\begin{proposition}
\label{prop:projectivegram}
Let $P$ be a finitely generated projective $R$-module, and let
$g:P\times P\to R$ 
be a symmetric bilinear form whose quadratic form is positive definite on
every real fiber.  Then there exist finitely many
$\ell_1,\dots,\ell_N\in P^*$ such that
\[
g(z,z)=\sum_{j=1}^N\ell_j(z)^2 \qquad \text{for all }z\in P.
\]
\end{proposition}

\begin{proof}
Choose a finitely generated projective module $Q$ with $P\oplus Q\cong R^r$. 
Choose generators
$\mu_1,\dots,\mu_t$ of $Q^*$, and define
\[
\gamma(y)=\sum_{j=1}^t\mu_j(y)^2.
\]
The form $\gamma$ is positive definite on every real fiber of $Q$: the
reductions of the $\mu_j$ generate the dual of each fiber, so they cannot all
vanish on a nonzero fiber vector.

\def\operp{\oplus}

Under an isomorphism $P\oplus Q\cong R^r$, the quadratic form $g\operp\gamma$ is given by a
symmetric matrix $H\in\Sym{r}{R}$, which is positive definite at every $p\in X(\R)$. Since $X(\R)$ is compact, there is $\rho\in\R$ such that $\rho^2-\sum_{i=1}^n x_i^2$ is strictly positive on $X(\R)$. Thus, $\rho^2-\sum_{i=1}^n x_i^2$ is a sum of squares in $R$ by \cite[Corollary 3]{Schmudgen1991}. This archimedean relation allows one to apply the matricial version of Putinar's Positivstellensatz \cite[Theorem 2]{Scherer2006} to $H$, showing that $H=G^\Tt G$ for some $G\in \Mat{N,r}{R}$. 
Equivalently, $g\operp\gamma$ is a finite sum
of squares of linear functionals on $P\oplus Q$. Restricting these
functionals to the direct summand $P$ gives the asserted representation of
$g$.
\end{proof}

\subsection{Final step in the proof of Theorem~\ref{thm:main}}

Let $T$ and $L$ be as constructed in the previous subsections.  By \eqref{eq:TglobalL}, we view $T$ as an $R$-linear map $R^d\to L$. The quadratic form $h(z)=\langle z,z\rangle$ on $L$ is
positive definite on every real fiber by Lemma~\ref{lem:hpositive}.
Proposition~\ref{prop:projectivegram} therefore gives
$\ell_1,\dots,\ell_N\in L^*$ such that
\[
h(z)=\sum_{j=1}^N\ell_j(z)^2.
\]
For $v\in R^d$,
\[
v^{\mathsf T}Av
    =\langle Tv,Tv\rangle
    =h(Tv)
    =\sum_{j=1}^N(\ell_j\circ T)(v)^2.
\]
Each $\ell_j\circ T$ is an element of $(R^d)^*$, and hence has the form $(\ell_j\circ T)(v)=\sum_{i=1}^d b_{ji}v_i$ for some $b_{ji}\in R$. 
If $B=(b_{ji})$, then
\[
v^{\mathsf T}Av=v^{\mathsf T}B^{\mathsf T}Bv
\qquad\text{for all }v,
\]
and therefore $A=B^{\mathsf T}B$. \hfill$\qed$

\section{Corollaries}\label{s:cors}
 This section contains three corollaries of  Theorem~\ref{thm:main}. The first the matricial bivariate Fej{\'e}r–Riesz theorem; the second 
 answers an open problem on positivity in tensor products \cite{Mehta2026} that arose in quantum information theory; and the third answers 
 a question from \cite[Section 11]{Savchuk2012} by showing $\Mat{4}{\C[S^1\times S^1]}$ admits a perfect Positivstellensatz.

\begin{corollary}\label{c:mtx2FR}
If $A\in \Mat{n}{\C[S^1\times S^1]}$ is positive semidefinite on $S^1\times S^1$, then there are $m\in\N$ and $B\in \Mat{m\times d}{\C[S^1\times S^1]}$ such that $A=B^*B$.
\end{corollary}

Next, we apply Corollary \ref{c:mtx2FR} to an open problem on positivity in tensor products \cite{Mehta2026}. Therein, the authors consider $*$-algebras of the form $\cA_{m,n}:=\C[\Z_m^{*n}\times \Z_m^{*n}]$ for $m,n\ge 2$; these are relevant in quantum information science, as they are the algebras of observables in bipartite Bell scenarios. In \cite[Corollary 1.2]{Mehta2026}, they show that unless $m=n=2$, the algebra $\cA_{m,n}$ contains a positive semidefinite element (i.e., positive semidefinite in all $*$-representations of $\cA_{m,n}$), but is not a sum of hermitian squares in $\cA_{m,n}$. In \cite[Example 6.8]{Mehta2026}, they ask whether positive semidefinite elements in $\cA_{2,2}$ are sums of hermitian squares. 

\begin{corollary} \label{c:msz}
If $a\in \cA_{2,2}$ is positive semidefinite, then $a=\sum_ib_i^*b_i$ for some $b_i \in \cA_{2,2}$.
\end{corollary}

\begin{proof}
First, observe that $\Z_2*\Z_2\cong\Z\rtimes \Z_2$ (where $\Z_2$ acts on $\Z$ by changing the sign). Then, $(\Z_2*\Z_2)\times (\Z_2*\Z_2)\cong (\Z\times \Z)\rtimes (\Z_2\times\Z_2)$, with the component-wise action $\alpha$ of $\Z_2\times\Z_2$ on $\Z\times \Z$. Consequently, $\cA_{2,2}$ is a cross product
$$\cA_{2,2}\cong 
\C[S^1\times S^1]\times_\alpha (\Z_2\times\Z_2).$$
Note that the group $\Z_2\times\Z_2$ has order 4. By \cite[Proposition 5.4]{Savchuk2012}, there is a strong conditional expectation $\vartheta:\Mat{4}{\C[S^1\times S^1]}\to \cA_{2,2}$. That is, $\cA_{2,2}$ embeds into $\Mat{4}{\C[S^1\times S^1]}$, and $\vartheta$ is a unital map of $\cA_{2,2}$-bimodules that maps preserves sums of hermitian squares. If $a\in \cA_{2,2}$ is positive semidefinite, then $a$ is also positive semidefinite in $\Mat{4}{\C[S^1\times S^1]}$ by \cite[Lemma 5.5]{Savchuk2012}. Hence, $a$ is a sum of hermitian squares in $\Mat{4}{\C[S^1\times S^1]}$ by Corollary \ref{c:mtx2FR}. By applying $\vartheta$ to this decomposition, we conclude that $a$ is a sum of hermitian squares in $\cA_{2,2}$.
\end{proof}

Finally, in \cite[Section 11]{Savchuk2012}, the authors ask whether the $*$-algebra $\Mat{n}{\C[S^2]}$ admits a perfect Positivstellensatz (i.e., positive semidefinite elements are sums of hermitian squares). They also point out that a perfect Positivstellensatz for $\Mat{2}{\C[S^2]}$ implies a perfect Positivstellensatz for the ``noncommutative sphere''
$$\C\left\langle s_1,s_2,s_3\mid 
s_i^*=s_i \text{ for }i=1,2,3,\ 
s_1^2+s_2^2+s_2^3=1,\ 
s_is_j=-s_js_i \text{ for }i\neq j
\right\rangle.$$
Theorem \ref{thm:main} applies to $X=\{x_1^2+x_2^2+x_3^2=1\}$, whose real points form the sphere $S^2$, and we obtain the following.

\begin{corollary}\label{c:mtxsphere}
If $A\in \Mat{n}{\C[S^2]}$ is positive semidefinite on $S^2$, then $A$ is a sum of hermitian squares in $\Mat{n}{\C[S^2]}$.
\end{corollary}


\end{document}